\documentclass[12pt]{article}

\usepackage{enumerate}
\usepackage{amsmath}
\usepackage{amssymb,latexsym}
\usepackage{amsthm}
\usepackage{color}
\usepackage{graphicx}
\usepackage{hyperref}
\usepackage{amscd}

\title{Extending a result of Carlitz and McConnel to polynomials which are not permutations}
\author{Bence Csajb\'ok\thanks{
		The author is supported by the J\'anos Bolyai Research Scholarship of the Hungarian Academy of Sciences.}}
\date{}

\newcommand{\F}{{\mathbb F}}

\newtheorem{theorem}{Theorem}[section]

\newtheorem{corollary}[theorem]{Corollary}

\newtheorem{result}[theorem]{Result}

\DeclareMathOperator{\AG}{{AG}}

\begin{document}
\maketitle

\begin{abstract}
Let $D$ denote the set of directions determined by the graph of a polynomial $f$ of $\mathbb{F}_q[x]$, where $q$ is a power of the prime $p$. If $D$ is contained in a multiplicative subgroup $M$ of $\F_q^\times$, then by a result of Carlitz and McConnel it follows that $f(x)=ax^{p^k}+b$ for some $k\in \mathbb{N}$. Of course, if $D\subseteq M$, then $0\notin D$ and hence $f$ is a permutation. If we assume the weaker condition $D\subseteq M \cup \{0\}$, then $f$ is not necessarily a permutation, but Sziklai conjectured that $f(x)=ax^{p^k}+b$ follows also in this case. When $q$ is odd, and the index of $M$ is even, then a result of Ball, Blokhuis, Brouwer, Storme and Sz\H onyi combined with a result of McGuire and G\"olo\u{g}lu proves the conjecture. Assume $\deg f\geq 1$. We prove that if the size of $D^{-1}D=\{d^{-1}d' : d\in D\setminus \{0\},\, d'\in D\}$ is less than $q-\deg f+2$, then $f$ is a permutation of $\F_q$. We use this result to verify the conjecture of Sziklai.
\end{abstract}



\section{Introduction}

Let $\F_q$ denote the finite field of $q=p^n$ elements, where $p$ is a prime. 
If $f$ is an $\F_q \rightarrow \F_q$ function then the affine $q$-set $\{(x,f(x)):x\in \F_q\}\subseteq \AG(2,q)$ is called the graph of $f$, the subset $D_f:=\{(f(x)-f(y))/(x-y) : x,y\in \F_q, x\neq y\}$ of $\F_q$ is called the set of directions, or slopes, determined by (the graph of) $f$. 
Each $\F_q \rightarrow \F_q$ function can be uniquely represented by a polynomial of $\F_q[x]$ and of degree at most $q-1$, so we will consider polynomials instead of function. For a subset $A$ of $\F_q$ we will denote $A\setminus \{0\}$ by $A^*$. 

When $d>1$ is a divisor of $q-1$ then put 
\[M_d := \{x^d : x \in \F_q^*\}.\] 
The following result was proved first by Carlitz in the case of $d=2$ \cite{Carlitz} and then generalised by McConnel for other divisors $d$ of $q-1$ \cite{McConnel}. 

\begin{result}
	\label{res1}
	If $D_f \subseteq M_d$, then $f$ is of the form $f(x) = a + bx^{p^k}$.
\end{result}

There are several generalisations and different proofs of this result, due to Bruen and Levinger \cite{BL}, Grundh\"ofer \cite{nemet}, Lenstra \cite{Lenstra}, see  \cite[Section 9]{Jones} for a survey on these results and their relation with the Paley graph.

Since $0 \notin D_f$ yields $f(x)\neq f(y)$ for each $x\neq y$, it is obvious that only permutations can satisfy the condition $D\subseteq M_d$. This is not the case anymore if we allow $D \subseteq M_d \cup \{0\}$. In \cite[pg. 114]{polybook} Sziklai conjectured that $f(x)=ax^{p^k}+b$ holds also when one replaces $M_d$ with $M_d \cup \{0\}$ in the statement of Result \ref{res1}. To present what is known regarding this conjecture, we need to recall some definitions. An $\F_q \rightarrow \F_q$ function $f$ is called additive if $f(a+b)=f(a)+f(b)$ for each $a,b\in \F_q$. Such functions correspond to polynomials
$a_0x+a_1x^p+\ldots+a_{n-1}x^{p^{n-1}} \in \F_q[x]$. 
If $f$ is additive and $\alpha \in \F_q$, then we we will call $f+\alpha$ an affine polynomial. An important result on directions is the following, due to Ball, Blokhuis, Brouwer, Storme, Sz\H{o}nyi \cite{BBBSSz} and Ball \cite{Ball}.

\begin{result}
	\label{res2}
	If $|D_f|\leq (q+1)/2$, then $f$ is an affine polynomial.
\end{result}

If $f$ is additive, then for the affine polynomial $g=f+\alpha$ it holds that $D_g=\{f(x)/x : x\in \F_q^*\}=D_f$. The following result is due to McGuire and G\"olo\u{g}lu \cite{McGuire}.

\begin{result}
	\label{res3}
	If $q$ is odd, $f$ is additive and $D_f \subseteq M_2 \cup \{0\}$, then $f(x)=ax^{p^k}$.
\end{result}

If $D_g \subseteq M_2 \cup \{0\}$, then $|D_g|\leq (q+1)/2$ and hence by Result \ref{res2} $g=f+\alpha$ for some additive $f$ with $D_f=D_g$, thus Result \ref{res3} proves Sziklai's conjecture in case of odd $q$ and even $d$. In \cite{McGuire} the authors used Kloosterman sums and it seems to be that their technique cannot be used to prove the conjecture when $d$ is odd.

\medskip

Our key result is Theorem \ref{main}. It's proof was inspired by the recent manuscript \cite{kyle} of C. H. Yip who used Result \ref{res2} in a clever way to strengthen Result \ref{res1}. In Corollary \ref{cor2} we extend Yip's result to polynomials which are not necessarily permutations. 
Our notation is standard, if $A,B \subseteq \F_q$ then  $A^{-1}=\{a^{-1} : a \in A^*\}$ and $AB=\{ab : a\in A, b\in B\}$. If $c\in \F_q$, then $A-c=\{a-c : a \in A\}$. 
In Theorem \ref{main2} and Corollary \ref{cor1} we present conditions on the size of $D_f^{-1}D_f$ which ensure that $f$ is a permutation. In Corollary \ref{conj} we prove Sziklai's conjecture. 

\section{New Results}

Our first result shows how a simple combinatorial property of the graph of $f$ implies an algebraic property of $D_f$.

\begin{theorem}
	\label{main}
	Assume that the line of equation $y=mx+b$ meets the graph of $f\in \F_q[x]$ in $k$ points for some $1<k<q$. Then
	\[|(D_f-m)^{-1} (D_f-m)|\geq q-k+2.\]
\end{theorem}

\begin{proof}
%
	Put $g(x):=f(x)-mx-b$. Then $g$ has exactly $k$ distinct roots. Let $a$ be one of them and define $h(x):=g(x+a)$. Clearly 
	$h$ has exactly $k$ distinct roots and $0$ is one of them. Also,
	
	\[D_h=\left\{\frac{h(x)-h(y)}{x-y} : x,y\in \F_q, x\neq y\right\}=\]
	\[\left\{\frac{f(x+a)-f(y+a)-m(x-y)}{(x+a)-(y+a)} : x,y\in \F_q, x\neq y\right\}=D_f-m.\]
	Denote by $r_1,r_2,\ldots,r_{k-1}$ the distinct non-zero roots of $h$. 
	
	We claim that
\begin{equation}
		\label{eq1}
H:=\left\{ \frac{x}{x-y} :x,y\in \F_q,\, h(x)\neq 0,\, h(y)=0 \right\} \subseteq D_h^{-1} D_h,
\end{equation}
	and 
\begin{equation}
	\label{eq2}
	|H|\geq q-k+1.	
\end{equation}
To prove \eqref{eq1} take an element $x/(x-y)$ from $H$ (i.e., $h(x)\neq 0$, $h(y)=0$) and put $c=x/h(x)$. Since $h(0)=0$, we have
\[c=\frac{x}{h(x)}=\frac{x-0}{h(x)-h(0)}\in D_h^{-1}.\]
Also, since $h(y)=0$, we have
\[\frac{x}{x-y}=c \frac{h(x)-h(y)}{x-y}\in D_h^{-1}D_h.\]
To prove \eqref{eq2} first note that $0,1\in D_h^{-1}D_h$, $0\notin H$ and $1 \in H$. 
Assume that $\alpha\in \F_q\setminus \{0,1\}$ is not contained in $H$. Then the solutions in $x$ of the equations 
\begin{equation}
	\label{eq3}
x/(x-r_i)=\alpha, \quad i\in \{1,2,\ldots, k-1\},
\end{equation}
are in  $\{r_1,r_2,\ldots,r_{k-1}\}$.
If $x$ is a solution of $x/(x-r_i)=\alpha$, then
$x=r_i \alpha/(\alpha-1)$, so if $\alpha\notin H$, then
\[\alpha/(\alpha-1)\{r_1,r_2,\ldots,r_{k-1}\}\subseteq\{r_1,r_2,\ldots,r_{k-1}\},\]
and hence 
\[\alpha/(\alpha-1)\{r_1,r_2,\ldots,r_{k-1}\}=\{r_1,r_2,\ldots,r_{k-1}\}.\]
Put $\beta=\alpha/(\alpha-1)$ and $R=\{r_1,r_2,\ldots, r_{k-1}\}$. Then $\beta R=R$. Clearly there are at most $k-1$ such $\beta$s: $r_1/r_1,r_2/r_1,\ldots, r_{k-1}/r_1$, but $\beta\neq 1$ and hence there are at most $k-2$ values of $\alpha\in \F_q \setminus \{0,1\}$ for which \eqref{eq3} does not have a solution with $h(x)\neq 0$. It follows that $|H|\geq (q-2)-(k-2)+1$ (because of $1\in H$).  
Since $0\in D_h^{-1}D_h\setminus H$, we have  $|D_h^{-1}D_h|\geq |H|+1$, which proves the assertion.
\end{proof}

\begin{theorem}
	\label{main2}
	Let $f\in \F_q[x]$ be a polynomial of degree $k$ for some $0<k<q$. If $|D_f^{-1}D_f| < q-k+2$, then 
	$f$ is a permutation of $\F_q$.
\end{theorem}
\begin{proof}
	If $k=1$, then $f$ is a permutation. Assume $k>1$.
	If $f$ was not a permutation, then there would be a line with slope $0$ meeting $U_f$ in at least two and at most $k$ points (since $f(x)=c$ has at most $\deg f$ roots for each $c\in \F_q$). By Theorem \ref{main} it follows that $|D_f^{-1}D_f|\geq q-k+2$, contradicting the assumption. 
\end{proof}

\begin{corollary}
	\label{cor1}
	If $|D_f^{-1}D_f| \leq (q+1)/2$ holds for some $\F_q \rightarrow \F_q$ non-constant function $f$, then $f$ permutes $\F_q$. 
\end{corollary}
\begin{proof}
	$f$ is non-constant, hence there is a non-zero element in $D_f$, so $D_f^{-1}\neq \emptyset$ and 
	$|D_f| \leq |D_f^{-1}D_f|$. By Result \ref{res2} $f$ can be represented by an affine polynomial of degree at most $q/p\leq q/2$. Then  
	\[|D_f^{-1}D_f|\leq \frac{q+1}{2}<q-q/2+2\leq q-\deg f +2,\]
	and the result follows from Theorem \ref{main2}.
\end{proof}

The next result proves Conjecture 18.11 form \cite[pg. 114]{polybook}.

\begin{corollary}
	\label{conj}
	If $D_f \subseteq M_d \cup \{0\}$, then $f(x)=ax^{p^k}+b$.
\end{corollary}
\begin{proof}
	$D_f^{-1}D_f \subseteq M_d^{-1}(M_d\cup \{0\})=M_d \cup \{0\}$ and hence $|D_f^{-1}D_f|\leq |M_d|+1\leq (q+1)/2$. By Corollary \ref{cor1} $f$ is a constant function (that is, $a=0$), or it is a permutation and hence $0 \notin D_f$. The assertion follows from Result \ref{res1}.
\end{proof}

The next result weakens the condition on $f$ from \cite[Theorem 1.2]{kyle}, since we do not require $f$ to be a permutation.

\begin{corollary}
	\label{cor2}
	If for some $\F_q \rightarrow \F_q$ function $f$ it holds that $|D_f^{-1} D_f D_f^{-1}|\leq (q+1)/2$, then $f(x)=ax^{p^k}+b$.
\end{corollary}
\begin{proof}
	If $f$ is the constant function then the result trivially holds. If it is not, then $D_f^{-1}\neq \emptyset$ and hence $|D_f^{-1}D_f|\leq |D_f^{-1}D_fD_f^{-1}|\leq (q+1)/2$, thus $f$ is a permutation by Corollary \ref{cor1}. Then the result follows from \cite[Theorem 1.2]{kyle} and from the fact that $1\in D_f^{-1}D_f$ and hence $D_f^{-1}\subseteq D_f^{-1}D_fD_f^{-1}$, thus $D_f^{-1}\cup D_f^{-1}D_fD_f^{-1}=D_f^{-1}D_fD_f^{-1}$. 
\end{proof}

\medskip

\begin{flushleft}
	Bence Csajb\'ok \\
	ELTE E\"otv\"os Lor\'and University, Budapest, Hungary\\
	Department of Computer Science\\
	1117 Budapest, P\'azm\'any P.\ stny.\ 1/C, Hungary\\
	{{\em bence.csajbok@ttk.elte.hu}}
\end{flushleft}


\begin{thebibliography}{pippo}
	
\bibitem{Ball}
{\sc S. Ball:} The number of directions determined by a function over a finite field, J.\ Combin.\ Theory Ser.\ A,
104(2) (2003), 341--350.

\bibitem{BBBSSz}
{\sc A. Blokhuis, S. Ball, A. E. Brouwer, L. Storme, and T. Sz\H onyi:} On the number of slopes of the graph of a function defined on a finite field, J.\ Combin.\ Theory Ser.\ A, 86(1) (1999), 187--196.

\bibitem{BL} 
{\sc A. Bruen and B. Levinger:} A theorem on permutations of a finite field, Canadian J. Math.\ 25 (1973), 1060--1065.

\bibitem{Lenstra}
{\sc H. W. Lenstra, Jr.:} Automorphisms of finite fields. J. Number Theory, 34(1) (1990), 33--40.

\bibitem{Carlitz}
{\sc L. Carlitz:} A theorem on permutations in a finite field. Proc. Amer. Math. Soc. 11 (1960), 456--459.

\bibitem{McConnel}
{\sc R. McConnel:} Pseudo-ordered polynomials over a finite field, Acta Arith.\ 8 (1963), 127--151.

\bibitem{McGuire}
{\sc G. McGuire and F. G\"olo\u{g}lu:} On theorems of Carlitz and Payne on permutation polynomials over finite fields with an application to $x^{-1} + L(x)$, Finite Fields Appl. 27 (2014), 130--142.

\bibitem{nemet}
{\sc T. Grundh\"ofer:} \"Uber Abbildungen mit eingeschr\"anktem Differenzenprodukt auf einem endlichen K\"orper. Arch.
Math. (Basel), 37(1) (1981), 59--62.

\bibitem{Jones}
{\sc G. A. Jones:} Paley and the Paley graphs. In Isomorphisms, symmetry and computations in algebraic graph theory, volume 305 of Springer Proc. Math. Stat., pages 155--183. Springer, Cham, 2020.



\bibitem{polybook}
{\sc P. Sziklai:} Polynomials in finite geometry, 2008, 
\href{https://www.academia.edu/69422637/Polynomials\_in\_finite_geometry}{https://www.academia.edu/69422637/Polynomials\_in\_finite\_geometry}

\bibitem{kyle}
{\sc C. H. Yip:} A strengthening of McConnel's theorem on permutations over finite fields, 2024, 
\href{https://arxiv.org/abs/2407.21362}{https://arxiv.org/abs/2407.21362}


\end{thebibliography}
\end{document}